\documentclass[12pt,letterpaper]{amsart}
\usepackage{mathrsfs,amsmath,amssymb,amsthm,amsfonts}
\usepackage{setspace,color,geometry,mathpazo}
\usepackage{color}

\newtheorem{Theorem}{Theorem}[section]
\newtheorem{Definition}[Theorem]{Definition}
\newtheorem{Proposition}[Theorem]{Proposition}

\newtheorem{Lemma}[Theorem]{Lemma}
\newtheorem{Corollary}[Theorem]{Corollary}
\theoremstyle{remark}
\newtheorem{Example}[Theorem]{Example}
\newtheorem{Remark}[Theorem]{Remark}

\def\ovr{\overline}
\def\Om{\Omega}
\def\al{\alpha}

\def\ve{\varepsilon}
\def\sbs{\subset}
\def\up{\stackrel}
\def\ra{\rightarrow}

\def\co{\operatorname{co}}

\def\ext{\operatorname{ext}}

\def\lim{\operatorname{lim}}
\def\ic{\operatorname{int}}
\def\re{\mathbf {Re\,}}
\def\be{\begin{enumerate}}
\def\ee{\end{enumerate}}
\def\bT{\begin{Theorem}}
\def\eT{\end{Theorem}}
\def\bP{\begin{Proposition}}
\def\eP{\end{Proposition}}
\def\bD{\begin{Definition}}
\def\eD{\end{Definition}}
\def\bE{\begin{Example}}
\def\eE{\end{Example}}
\def\bL{\begin{Lemma}}
\def\eL{\end{Lemma}}
\def\bC{\begin{Corollary}}
\def\eC{\end{Corollary}}

\def\I{{\mathcal I}}

\def\B{{\mathcal B}}

\def\J{{\mathcal J}}

\def\K{{\mathcal K}}

\def\R{{\mathbb R}}

\usepackage[active]{srcltx}

\date{\today}
\begin{document}
\title{Lower Envelopes and Steepest Descent Directions in Vector
Optimization}
\author{N\.{i}hat G\"{o}khan G\"{o}\u{g}\"{u}\c{s}}
\address[N\.{i}hat G\"{o}khan G\"{o}\u{g}\"{u}\c{s}]{Sabanc\i{} University,
Orhanli Tuzla, 34956, Istanbul, Turkey}
\email{nggogus@sabanciuniv.edu}

\thanks{The author is supported by the Scientific and Technological Research Council
of Turkey related to a grant project called "Jensen measures in
complex analysis and c-regularity" with project number 110T223.}


\begin{abstract} The purpose of the paper is to give a complete
characterization of the continuity of lower envelopes in the
infinite dimensional spaces in terms of the notion of
$c$-regularity. As an application we introduce a variational
unconstrained vector optimization problem for smooth functions and
characterize when the variational steepest descent directions are
continuous in terms of the generating sets which are considered.

\end{abstract}  \maketitle

\section {Introduction}
The lower envelopes  of certain functions appear quite naturally in
functional analysis, optimization, in the theory of uniform algebras
and in potential theory. We investigate the continuity properties of
lower envelopes in the abstract setting of infinite dimensional
spaces. One can start with any set in a topological space $A$ and
assign to each point $a$ in this set \emph{a fiber} $\J_a$, that is,
a class of elements from the dual space $X^*$ of some vector space
$X$. Then one can construct a new function on $A$ by taking lower
envelopes which is obtained by considering the infimum over all
numbers of the form $\re x^*(x)$, where $x\in X$ is fixed, and $x^*$
changes over the fiber $\J_a$ for any $a\in A$. To visualize things,
as a model example one can think of $x$ as a function which we
minimize subject to some condition $\J_a$, where $a$ runs in some
sample space $A$. Then we wish to find conditions which guarantee
continuity of these optimal values at a point $a\in A$. We consider
fibers as multifunctions. As it happens the continuity of lower
envelopes is a consequence of such geometric properties of these
multifunctions as upper and lower semicontinuity (Theorem
\ref{Th:Mf}). Roughly speaking, lower envelopes are continuous if
and only if any limit point of fibers can be obtained as a limit of
all fibers from every direction. In Section \ref{Sec:Creg} we call
such sets $c$-{\it{regular}}. This notion was introduced first in
\cite{Gogus05} and \cite{GogusThesis} in the content of
pluripotential theory for domains in ${\mathbb{C}}^n$.

\par As an application of this characterization we look at the
problem of unconstrained $K$-minimizers. In multi-objective
optimization, a special case of the problem of unconstrained
$K$-minimizers, one considers a continuously differentiable function
\[F:\R ^n\to \R^m.\] The problem is to find a minimizer of $F$ on
$\R^n$ subject to the convex cone $\R^m_+$ of positive octant in
$\R^m$. To explain further let
\[\R^m_+=\{(x_1,\ldots ,x_m)\in\R^m:x_j\geq 0,\,\,j=1,\ldots ,m\}\]
and we want to find a point $\al\in\R^n$ such that there exists no
other point $\beta\in\R^n$ with $F(\beta)\not =F(\al)$ and
$F(\al)-F(\beta)\in \R^m_+$. Recently, this problem was extended
using the Cauchy method (or known as steepest descent method),
Newton method and gradient projection method to the problem of
finding unconstrained $K$-minimizers in \cite{DrummondGranaIusem04},
\cite{FliegeGranaDrummond09}, \cite{FliegeSvaiter00}, and
\cite{DrummondSvaiter05}. To find the $K$-minimizers one needs to
look at the $K$-critical points of $F$. As in the scalar case $m=1$
every $K$-minimizer is $K$-critical but not vice versa. The method
of $K$-steepest descent in \cite{DrummondSvaiter05} provides an
efficient algorithm to approximate the $K$-critical points. A
central tool of these investigations is the so called gauge function
$G(x)$ for $K$. It allows one to measure how good the descent
direction is.

\par In this paper we describe this problem in the abstract setting of
infinite dimensional spaces taking into account a family of the
minimization sets and a family of objective values. The number of
sets and objective values we consider is not necessarily finite. We
start with a family of closed convex pointed cones $K_a$, $a\in A$,
in a normed linear space $X$, where $A$ is a metric space. Let
$\J_a$ be a generating set for $K_a$. On our way we consider the
variational gauge function $G(a,x)$ for $K_a$ defined on $A\times
X$.  Using our characterization of $c$-regular sets from section
\ref{Sec:LowerEnvelopes} we completely characterize in section
\ref{Sec:KMinimizers} the continuity of $G(a,x)$ in terms of the
generating sets $\J_a$ under very reasonable conditions on $\J_a$.
We note that when $E=\R^n$, and $A$ and $S$ are one-point sets, we
are in the same consideration as in the work
\cite{DrummondSvaiter05} and in this case the continuity of the
gauge function is trivial.

\par Let $E=\R^n$ and consider a family \[F_s:\R^n\to X,\,\,\,\,\,\, s\in S,\] of continuously differentiable
functions indexed by some topological space $S$. So for each $a\in
A$, $s\in S$, and $\al\in \R^n$ one can find the $K_a$-steepest
descent direction $\nu [a,s,\al]$ for $F_s$ at $\al$ as described in
\cite{DrummondSvaiter05}. Since $K_a$-steepest descent directions
are used to approximate the $K_a$-critical values for the functions
$F_s$ it is important to characterize when the functions $\nu
[a,s,\al]$ are continuous. We prove in section \ref{Sec:KMinimizers}
that if the differential maps of $F_s$ are continuous in $s$ and if
the index set $A$ is $\J$-c-regular, then the $K_a$-steepest descent
direction $\nu [a,s,\al]$ for $F_s$ at $\al$  are continuous.


\section{$C$-regularity}\label{Sec:Creg}
\par Let $A$ be a metric space and $X$ be a norm space. To each element $a\in A$ we associate
a set $\J_a\sbs X^*$. We will use the notation $x_j^*\up{* }{\ra}
x^*$ when $x_j^*$ is a sequence in $X^*$ which converge weak-$*$ to
$x^*$. Given any point $a\in A$ let $S_a$ be the class of all
sequences $s=\{a_j\}$ in $A$ which converge to $a$. If $s\in S_a$,
then the set $\J^s_a$ consists of all elements $x^*\in X^*$ so that
a sequence of elements $x_j^*\in \J_{a_j}$ converges weak-$*$ to
$x^*$. We denote by $\J^{ws}_a$ the set of all weak-$*$ cluster
points of $\J_{a_j}$ consisting of all elements $x^*\in X^*$ so that
there exist a subsequence $\{a_{j_k}\}$ of $s$ and elements $x
_{j_k}^*\in \J_{a_{j_k}}$ which converge weak-$*$ to $x^*$. Let
$\J^{cs}_a := \ovr{\co } \J^{ws}_a$, the closed convex hull of
$\J^{ws}_a$. We let
\[\J^1_a:=\cup _s\J^s_a, \,\,\J^{2s}_a:=\cap
_s\J^{s}_a, \,\,\J^{2ws}_a:=\cap _s\J^{ws}_a,\,\,\,\,\,
\textnormal{and} \,\,\,\,\,\J^{2cs}_a:=\cap _s\J^{cs}_a,\] where $s$
runs through all sequences in $S_a$. We will always assume that the
following properties hold:

$\J 0$: For each $a\in A$ the set $\J_a$ is a nonempty convex
weak-$*$ compact subset of $X^*$;

$\J 1$: For any convergent sequence $\{a_j\}$ in $A$  if $x^*_j\in
\J_{a_j}$, then there exists a subsequence $x^*_{j_k}$ that weak-$*$
converges;

$\J 2$: The set $\J^{2s}_a$ is nonempty for any $a\in A$;

$\J 3$: The set $\J^{cs}_a$ is weak-$*$ compact for any $a\in A$ and
$s\in S_a$.

By Alaoglu's theorem the conditions $\J 0$, $\J 1$, and $\J 3$ are
satisfied for example when $\cup_{a\in A}\J_a$ is bounded in $X^*$.

\begin{Remark}\label{Rem:Jcs} By principle of uniform boundedness
(see \cite[III. 14, Theorem 14.1]{Conway90}) if the condition $\J 3$
is satisfied, then for every $a\in A$ there exists a constant
$c_a>0$ so that $||x^*||\leq c_a$ for all $x^* \in \J ^{cs}_a$ .
\end{Remark}

\par Note that the sets $\J^s_a$ and $\J^{2s}_a$ are convex.
Since $\J^{ws}_a$ is the set of weak-$*$ cluster points of
$\J_{a_j}$'s, it is weak-$*$ closed. Hence $\J^{ws}_a$ is weak-$*$
compact. It follows from these observations that the sets
$\J^{2ws}_a$ and $\J^{2cs}_a$ are convex and compact for every $a\in
A$. It is not hard to see that for any $s\in S_a$,
\[\J^{2s}_a\sbs\J^s_a\sbs \J^{ws}_a,\,\,\,\,\, \textnormal{ and }\,\,\,\,\,\J^{2s}_a\sbs \J^{2ws}_a\sbs
\J^{ws}_a\sbs \J^1_a.\] A point $a \in A $ is said to be {\it
 $\J$-$c$-regular } if $\J^1_a=\J^{2s}_a$. $A$ is said to be
$\J$-$c$-regular if every point $x\in A$ is $\J$-$c$-regular.

\begin{Remark}\label{Rem:Creg} The classes $\J^s_a$ are independent of the sequence $s$ if and
only if $a$ is $c$-regular. In this case $\J^1_a=\J^{2s}_a=\J_a$ is
convex and compact.
\end{Remark}

\par A point $a \in A$ is said to be {\it $c_1$-regular }
({\it $c_2$-regular}, resp.) if the classes $\J^{ws}_a$
($\J^{cs}_a$, resp.) are independent of the sequence $s\in S_a$. We
will first show that all different types of "{\it{c-regular}}"
definitions above are equivalent. We state this problem in terms of
functional analysis and we prove this equivalence in this general
format.

If $\mathcal{L}$ is any subset of a linear space $\mathcal{X}$, the
closed convex hull of $\mathcal{L}$ is denoted by $\ovr{\co }\,
\mathcal{L} $. For a compact convex subset of a normed linear space
$\mathcal{X}$, we denote by $\ext \, \mathcal{K}$ the set of all
extreme points of $\mathcal{K}$.

\bT \cite[V. 7 Theorem 7.8]{Conway90} \label{c2} Let $\mathcal{K}$
be a compact convex subset of a locally convex linear space
$\mathcal{X}$, and $\mathcal{L}$ be any subset of $\mathcal{K}$. If
$\ovr{\co }\, \mathcal{L} = \mathcal{K}$, then $\ext \, \mathcal{K}
\sbs \ovr{\mathcal{L}}$. \eT

 Let $K=\{K_j\} $ be a sequence of sets in a locally convex linear space $\mathcal{X}$. We
define:

\par $l(K) = \{x: \, x = \lim x_j \, \textnormal{ for } \, x_j \in
K_j \}$, all limit points of $K_j$;

\par  $w(K) = \{x :\, x = \lim x _{j_m} \, \textnormal{ for some
 subsequence } \, x_{j_m} \in K_{j_m} \}$, all
cluster points of $K_j$;

\par  $cw(K) = \ovr{\co } \, w(K)$, closed convex hull of $w(K)$.

The following result was proved in \cite{Gogus05} and
\cite{GogusThesis}. \bT \cite{GogusThesis} \label{Th:5} Let
$K=\{K_j\} $ be a sequence of compact convex sets in a locally
convex linear space $\mathcal{X}$ so that $cw(K)$ is also compact.
Suppose for any subsequence $L=\{L_j\}$ of $\{K_j\}$, $cw(L)=cw(K)$.
Then $l(L)=w(L)=cw(K)$ for all subsequences $L$ of $K$. \eT

The above theorem allows us to show the equivalence of different
$c$-regularities defined above.

\bC\label{ay3} A point $a \in A$ is $c$-regular if and only if it is
$c_1$-regular if and only if it is $c_2$-regular. \eC
\begin{proof}
As noted before the classes $\J^s_a$ are independent of the sequence
$s$ if and only if $a$ is $c$-regular. It's easy to see that
$$c\text{-regular}\Rightarrow c_1\text{-regular}\Rightarrow c_2\text{-regular}$$
using the definitions.

To show that $c_2$-regularity implies $c$-regularity, let
$s=\{a_j\}$ be any sequence converging to $a$. Put $K =\{
\J_{a_j}\}$ in Theorem \ref{Th:5}. Then $l(K)=\J^s_a$,
$w(K)=\J_a^{ws}$ and $cw(K)=\J^{cs}_a$. Recall that $cw(K)$ is
compact in $X^*$. By Theorem \ref{Th:5}, $\J^s_a=\J^{cs}_a$ for any
sequence $s\in S_a$ and thus $\J^s_a$ is independent of $s$.
Therefore $a$ is $c$-regular.
\end{proof}

\section{Lower envelopes}\label{Sec:LowerEnvelopes}
\par Given any element $x$ in $X$, we
define its $\J$-envelope $\I [x,\J]:A\to \mathbb{R}$ as
$$\I [x;\J](a):=\inf \left\{ {\re} x^*(x) :x^* \in
\J_a \right\}$$ for every $a\in A$. Let us write $\I x(a)$ instead
of $\I [x;\J](a)$ for simplicity if no confusion arise. In this
section we will prove that $c$-regular points are exactly those
where the $\J $-envelopes are continuous. Let

\[\I^{\sharp}x(a):=\inf \left\{ {\re} x^*(x) :x^* \in \J^{\sharp}_a
\right\},\] where $\sharp$ is one of $1,\, s,\, ws,\, cs,\, {2s}, {2ws}
\text{ or } {2cs}$ for any $x\in X$, $a\in A$ and $s\in S_a$. We
will leave the details of the following observation.
\begin{Remark} \label{r2} $\I^ {ws}x(a)=\I ^{cs}x(a)$ for any $x\in X$, $a\in A$ and
$s\in S_a$.
\end{Remark}

The following result will be of great use.

\bP\label{Prop:I1} Let $x\in X$ and $a \in A$. Then there exist
sequences $s_0$, $t_0\in S_a$ so that\be
\item $(\I x)_*(a )=\I^1x(a)= \inf _{s\in S_a}\I^sx(a)=
\I^{t_0}x(a)$;

\item $(\I x)^*(a)=\sup_{s\in S_a}\I^{cs}x (a)= \I^{s_0}x
(a)\leq\I^{2s}x (a)$. \ee
 \eP

\begin{proof}
 Take $a \in A$ and suppose $\I x (a_j)<\I^1x(a)-\varepsilon $ for
 some sequence of points $a_j\in A$ converging to $a$ and some number
 $\varepsilon >0$. We can find
 a sequence $x^* _j \in \J_{a_j}$ so that for all $j$,
 $${\re} x^* _j(x)<\I^1x(a)-\varepsilon  .$$
 There exists a subsequence $x^* _{j_k}$ such that $x^* _{j_k} \up{* }{\ra} x^* $
   for some $x^* \in \J^1_{a }$. Hence, letting $j_k\to \infty $,
   $${\re} x^* (x) \leq \I^1x(a)- \varepsilon .$$ On the other hand
$\I^1x(a) \leq {\re} x^*(x) $, which gives that
   $${\re} x^*(x) \leq \I^1x(a)- \varepsilon \leq {\re} x^*(x)- \varepsilon ,$$
a contradiction. Thus $(\I x )_*(a ) \geq \I^1x(a)$.

Suppose $\I^1x(a)+\varepsilon <(\I x )_*(a )$ for some point $a \in
A$ and some number $\varepsilon >0$. We may find an element $x^* $
of $\J^1_a$ so that $${\re} x^*(x) \leq \I^1x(a)+\varepsilon .$$
There exists a sequence $a_j\in A $ and $x^* _j \in \J_{a_j}$ such
that $a_j\to a $ and $x^* _j$ converges weak-$*$ to $x^*$. Since $\I
x (a_j)\leq {\re} x^*_j(x)$ for all $j$,
\begin{eqnarray}
(\I x )_*(a )  \leq  \lim _j {\re} x^* _j(x)= {\re} x^* (x)  \leq
\nonumber \I^1x(a)+\varepsilon <(\I x)_*(a ).
\end{eqnarray}
This contradiction proves that $\I^1x(a)=(\I x)_*(a)\leq \I^sx(a)$
for any $s\in S_a$. Take points $a_j\in A$ and elements $x^*_j\in
\J_{a_j}$ so that $t_0=\{a_j\}\in S_a$, $\lim\I x(a_j)=(\I x)_*(a)$
and $\I x(a_j)={\re} x^* _j(x)$ for every $j$. Passing to a
subsequence we may assume that $x^* _j$ converges weak-$*$ to an
element $x^* \in \J^{t_0}_a$. Then
\[\I^{t_0}x(a)\leq (\I x)_*(a)={\re} x^* (x)\leq
\I^{t_0}x(a).\] Hence $(\I x)_*(a)=\I^{t_0}x(a)$ and this finishes
the proof of the first part.

For the second part, if $s=\{b_j\}\in S_a$, then we can find
elements $y^* _j\in \J _{b_j}$ so that $\I x(b_j)={\re} y^* _j(x)$
for each $j$. A subsequence $\{y^* _{j_k}\}$ converges weak-$*$ to
some $y^* \in \J^{ws}_a$. Then
\[(\I x )^*(a)\geq \lim {\re} y^* _{j_k}(x)={\re} y^*(x) \geq \I^{ws}x (a).\]
Thus, $(\I x)^*(a)\geq \sup_{s\in S_a}\I^{cs}x (a)$.

On the other hand there exist points $a_j\in A$ converging to $a$ so
that $\lim \I x(a_j)=(\I x)^*(a )$. Let $t=\{a_j\}$. There exist
$y^*\in \J^{wt}_a$ and $y^*_{j_k}\in \J_{a_{j_k}}$ that weak-$*$
converge to $y^*$ so that
\[\I^{wt}x(a)= y^*(x)=\lim y^*_{j_k}(x)\geq \lim \I
x(a_{j_k})=(\I x)^*(a).\] Hence we get the first equality in (2).

\par To prove the second equality note that for every $j$ there
exists an element $x^*_j\in \J _{a_j}$ so that $\I x(a_j)=x^*_j(x)$.
There exists a subsequence $x^*_{j_k}$ that weak-$*$ converges to
some $x^* \in \J ^{s_0}_a$, where we set $s_0=\{a_{j_k}\}$. Then
\[(\I x)^*(a)=\lim \I x(a_{j_k})=x^*(x)\geq \I ^{s_0}x(a).\]
Now given $\varepsilon >0$, $\I ^{s_0}x(a)+\varepsilon \geq {\re}
z^*(x)$ for some $z^*\in \J^{s_0}_a$. There exist $z^*_k\in
\J_{a_{j_k}}$ that weak-$*$ converge to $z^*$.
\[\I^{s_0}x (a)+\varepsilon \geq \lim {\re} z^*_k (x)\geq \lim \I x(a_{j_k})
=(\I x)^*(a ).\] Hence $\I^{wt}x(a)=\I^{s_0}x(a)=(\I x)^*(a)=\sup
_{s\in S_a}\I ^{ws}x(a)$. The result follows from Remark \ref{r2}.
\end{proof}

Proposition \ref{Prop:I1} provides the following characterization of
continuity of lower envelopes in terms of $c$-regularity.

\bC \label{Cor:Jcreg} Let $a\in A$ be a point. We have the following
statements:

a. $\J^{2cs}_a=\J_a$ if and only if $\I x(a)=(\I x)^*(a)=\I
^{2cs}x(a)$ for any $x\in X$.

b. $\J_a=\J ^1_a$ if and only if $\I x(a)=(\I x)_*(a)=\I ^1x(a)$ for
any $x\in X$.

c. $\J^{2s}_a=\J^1_a$ if and only if $\I x$ is continuous at $a$ for
any $x \in X$. \eC

\begin{proof}
In general we have $\I x(a)\leq \sup_{s\in S_a}\I^{cs}x (a)$ since
the constant sequence $s=\{a\}\in S_a$. If $\J_a=\cap _{s\in
S_a}\J^{cs}_a=\J^{2cs}_a$, we have the equality $\I x(a)= \sup_{s\in
S_a}\I^{cs}x (a)=(\I x)^*(a)$. Conversely, suppose $\I x(a)=(\I
x)^*(a)$ for every $x\in X$. Suppose that there exists $x^*\in \J
_a\backslash \J^{cs}_a$ for some $s\in S_a$. There exist $x\in X$
and a number $r>0$ so that
\[\I x(a)=(\I x)^*(a)\leq x^*(x)<\I ^{cs}x(a)-r\leq \sup_{s\in S_a}\I^{cs}x (a)
=(\I x)^*(a),\] where Proposition \ref{Prop:I1} is used in the last
equality. The contradiction shows that $\J _a\sbs \J^{cs}_a$ for
every $s\in S_a$. Thus $\J_a=\J^{2cs}_a$. This proves part a. The
statement in part b. concerning $(\I x)_*(a)$ and $\J ^1_a$ is
proved similarly.

\par To prove the last statement about continuity we note that
$\J^{2s}_a=\J^1_a$ implies $\J^{2cs}_a=\J^1_a=\J_a$. From a. and b.
$\I x$ is continuous at $a$ for any $x \in X$. To prove the converse
suppose that $(\I x)_*(a)= \I x(a)= (\I x)^*(a) $ for every $x\in
X$. Then $\J^{2cs}_a=\J^1_a=\J_a$, $\J^1_a$ is closed and convex and
hence $\J^{cs}_a= \J^1_a$ for every $s\in S_a$. This means that $a$
is $c_2$-regular and hence $c$-regular by Corollary \ref{ay3}.
\end{proof}

\par Let $F$ and $G$ be topological spaces and let $p:\,F\times G\to
F$ be the projection. A set $\K\sbs F\times G$ is a
{\it{multifunction}} on $F$ if $p(\K)=F$ and for each $x\in F$ the
fiber $\K_x=\{y\in G: \, (x,\, y)\in \K\}$ is compact.

\par A multifunction $\K$ is {\it{upper semicontinuous}} at $x\in F$ if for
every neighborhood $V$ of $\K_x$ in $F\times G$ there is a
neighborhood $W$ of $x$ in $F$ such that $\K_y\sbs V$ when $y\in W$.
A multifunction $\K$ is {\it{lower semicontinuous}} at $x\in F$ if
for every $(x,y)\in \K_x$ and for every neighborhood $V$ of $(x,y)$
in $F\times G$ there is a neighborhood $W$ of $x$ in $F$ such that
$\K_y\cap V\ne\emptyset$ when $y\in W$. The following is a slightly
modified version of Theorem 3.2 proved in \cite{Gogus05}.

\bT \label{Th:Mf} Let $\J\sbs A\times X^*$ be a multifunction on $A$
with fibers $ \J_a$ at $a\in A$ so that conditions $\J 0$-$\J 3$ are
satisfied. Let $a_0\in A$. \be
\item The lower envelope $\I x$ is upper semicontinuous at $a_0$ for
all $x \in X$ if and only if $\J$ is lower semicontinuous at $a_0$.

\item The lower envelope $\I x$ is lower semicontinuous at $a_0$ for
all $x \in X$ if and only if $\J$ is upper semicontinuous at $a_0$.
\ee \eT

\begin{proof}

(1)  Suppose $\J$ is lower semicontinuous at $a_0$. Choose $x^* \in
\J_{a_0}$ such that $${\re} x^*(x) <\I x(a_0)+\frac{\varepsilon
}{2}$$ and let
$$V=\left\{x^* +y^* :\,|y^*(x) |
<\frac{\varepsilon }{2} ,\, \,\, y^* \in X^* \right\}.$$ There
exists a neighborhood $W$ of $a_0$ such that if $a\in W$ there
exists $x^* +y^* _a \in V\cap \J_a$. Then
$$\I x(a)\leq {\re} x^*(x)+ {\re} y^* _a(x)<{\re} x^*(x) +\frac{\varepsilon }{2}<\I x(a_0)+\varepsilon .$$
Hence $\I x$ is upper semicontinuous at $a_0$.

Now suppose $\J$ is not lower semicontinuous at $a_0 \in A$. Then we
can find an element $x^* \in \J_{a_0}$, a neighborhood $V$ of $x^* $
and a sequence $a_k \in A$ such that $a_k \to a_0$ and $\J_{a_k}
\cap V=\emptyset $. Thus $x^*\in \J_{a_0}\backslash \J
^{2ws}_{a_0}$. By Corollary \ref{Cor:Jcreg} there exists $x\in X$ so
that $\I x$ is not upper semicontinuous at $a_0$.

 (2) Suppose $\J$ is upper semicontinuous at $a_0$. Let $x\in X$ and
$$V=\J_{a_0}+\left\{y^* :\, | y^*(x)| <\varepsilon
\right\}.$$ There exists a neighborhood $W$ of $a_0$ such that if
$a\in W$, $\J_a \sbs V$. Hence for all $x^* \in \J_a$ there exists
$y^* \in \J_{a_0}$ such that

$${\re} y^*(x) -\varepsilon <{\re} x^*(x)
<{\re} y^*(x) +\varepsilon .$$ Taking infimum over $x^* \in \J_a$,
we get
$$\I x(a_0)-\varepsilon \leq \I x(a)$$ for all $a\in
W$. Thus $\I x$ is lower semicontinuous at $a_0$.

\par Suppose $\J$ is not upper semicontinuous at some point $a_0$.
There exist a sequence $\{a_j\} \sbs A$ converging to $a$, a
neighborhood $V$ of $0$ in $X^*$ and elements $x^* _j \in
\J_{a_j}\backslash (\J_{a_0}+V)$. There exist a subsequence $\{x^*
_{j_k}\}$ of $\{x^* _j\}$ that converges weak-$*$ to an element $x^*
\in X^*$. Then $x^*\in \J^1_{a_0}$ but $x^* \not \in \J_{a_0}$. By
Corollary \ref{Cor:Jcreg} there exists $x\in X$ so that $\I x$ is
not lower semicontinuous at $a_0$.
\end{proof}

If $\B ^*$ is an open ball of $X^*$ and $X$ is separable, then it is
known that $\B ^*$ is metrizable. In this case Corollary
\ref{Cor:Jcreg} can be improved in the following way.

\bC \label{Cor:usc2} Let $\J\sbs A\times X^*$ be a multifunction on
$A$ as in Theorem \ref{Th:Mf}. Suppose that $X$ is separable. Let
$a\in A$. Then $\I x$ is upper semicontinuous at $a$ for any $x \in
X$ if and only if $\J_a=\J ^{2ws}_a$. \eC

\begin{proof} Note that $\J^{2ws}_a=\J_a$ implies $\J^{2cs}_a=\J_a$ so
sufficiency follows from Corollary \ref{Cor:Jcreg}. To prove
necessity let $\mu \in \J_a$. By Remark \ref{Rem:Jcs} we may assume
that $\J ^{cs}_a$ is contained in some open ball $\B^*$ in $X^*$.
Let $B_k$ be the open ball of radius $1/k$ around $\mu$ in $\B ^*$.
Given a sequence $s=\{a_j\}\in S_a$. Since $\J$ is lower
semicontinuous at $a$ by Theorem \ref{Th:Mf}, for any $k$ there
exists $j_k\geq 1$ and $\mu _{j_k}\in \J _{a_{j}}$ for all $j\geq
j_k$. The sequence $\{\mu_{j_k}\}$ converges weak-$*$ to $\mu$ and
$\mu \in \J ^{ws}_a$. Hence $\J_a\sbs \J ^{2ws}_a$. The other
inclusion always holds. This finishes the proof.
\end{proof}

\par Now let us consider the function $\I [\cdot ,\cdot]:A\times X\to
\mathbb{R}$ defined by
\[\I [a,x]:=\I x (a)\] for every $a\in A$ and $x\in X$. It is an easy
fact that the function $\I [a,\cdot]$ is continuous in the second
variable $x$ when $a\in A$ is fixed. In fact, one can show that it
is Lipschitz continuous. Let us give the proof of this fact.

\begin{Proposition} The function $\I [a,\cdot]$ is Lipschitz continuous in the second
variable $x$ when $a\in A$ is fixed.
\end{Proposition}

\begin{proof} To see this, let $a\in A$ be fixed and take $x,$ $y\in X.$
Then there exists an element $\mu \in \J_{a}$ so that
\[\I (x+y)(a)=\re \mu(x+y)=\re \mu (x)+\re \mu (y)\geq \I x(a)+\I y(a).\]
From this inequality we have
\[\I [a,x]-\I [a,y]\leq -\I [a,y-x]\leq c\|x-y\|,\] where $c=\sup\{\|x^*\|:x^*\in
\J_a\}$. Note that $c$ is finite due to property $\J 0$. Hence by
symmetry
\[|\I [a,x]-\I [a,y]|\leq c\|x-y\|\] for every $x$, $y\in X.$
This proves the claim that $\I [a,\cdot]$ is Lipschitz continuous.
\end{proof}

\par It is not true in general that if a function $F:U\times
V\to\mathbb{R}$ defined on some set $U\times V$ is separately
continuous, then it is jointly continuous. For a simple example one
may take the function $F(x,y)=\frac{xy}{x^2+y^2}$ when $(x,y)\not
=(0,0)$ and $F(0,0)=0$ defined on $\mathbb{R}^2$. Then $F(x,\cdot)$
is continuous when $x\in\mathbb{R}$ is fixed, $F(\cdot ,y)$ is
continuous when $y\in\mathbb{R}$ is fixed, but $F$ is not continuous
at $(0,0)$. Our next result shows that for our lower envelope
operator $\I[\cdot,\cdot]$ being separately continuous is the same
as being jointly continuous. We will need the following lemma.

\begin{Lemma}\label{Lem:ConvergenceCrossProductDual} Let $X$ be a normed space, $\mu_j\in X^*$ be elements
which weak-$*$ converge to an element $\mu\in X^*$, and $x_j\in X$
be elements which converge to some element $x\in X$. Then the
numbers $\mu_j(x_j)$ converge to $\mu(x)$.
\end{Lemma}

\begin{proof} Note that we have
\[|\mu_j(x_j)-\mu(x)|\leq |\mu_j(x_j)-\mu_j(x)|+|\mu_j(x)-\mu(x)|\leq c\|x_j-x\|+|\mu_j(x)-\mu(x)|\]
for some constant $c>0$ for every $j$. By assumption of the lemma it
is clear that the right hand side converges to zero as $j\to\infty$.
\end{proof}

\begin{Proposition}\label{Prop:JointContinuity} Let $\J:A\times X\to \mathbb{R}$ be a
multifunction satisfying the properties $\J 0$-$\J 3$.
\begin{itemize}
\item[i.] $\I x$ is upper semicontinuous on $A$ for every $x\in X$ if and only if
$\I [\cdot,\cdot]$ is upper semicontinuous on $A\times X$.

\item[ii.] $\I x$ is lower semicontinuous on $A$ for every $x\in X$ if and only if
$\I [\cdot,\cdot]$ is lower semicontinuous on $A\times X$.

\item[iii.] $\I x$ is continuous on $A$ for every $x\in X$ if and only if
$\I [\cdot,\cdot]$ is continuous on $A\times X$.
\end{itemize}
\end{Proposition}

\begin{proof} iii. follows from i. and ii. One direction in these statements
is trivial. We will only prove necessity. Let us start proving i.
Suppose $\I x$ is upper semicontinuous on $A$ for every $x\in X$.
Suppose on the contrary that $\I [\cdot,\cdot]$ is not  upper
semicontinuous at some point $(a,x)$ in $A\times X$. There exist
$(a_j,x_j)\in A\times X$ which converge to $(a,x)$, a number $\ve
>0$ and an element $\mu\in \J_a$ so that
\[\re\mu(x)+\ve\leq\I[a,x]+2\ve\leq \I[a_j,x_j]\] for every $j$.
Since by Theorem \ref{Th:Mf} $\J$ is lower semicontinuous at $a$,
there exists a subsequence $\{a_{j_k}\}$ and measures $\nu_{j_k}\in
\J_{j_k}$ so that $|\mu (x)-\nu_{j_k}(x)|<1/k$ for every $k\geq 1$.
By property $\J 1$ we may assume without loss of generality  by
passing to another subsequence if necessary that $\nu_{j_k}$
weak-$*$ converges to some measure $\nu\in X^*$. Then we have
\[\re \mu (x)+\ve\leq \re \nu_{j_k}(x)\] for every $k\geq 1$.  As
$k\to\infty$ we get
\[\re \mu (x)+\ve\leq \re \nu (x)= \re \mu (x)\] which is clearly a
contradiction. This proves (the necessity of) part i.

\par Now let us prove part ii. Suppose now that $\I [\cdot,\cdot]$ is not  lower
semicontinuous at some point $(a,x)$ in $A\times X$. There exist
$(a_j,x_j)\in A\times X$ which converge to $(a,x)$, a number $\ve
>0$ and elements $\mu_j\in \J_{a_j}$ so that
\[\re\mu_j(x_j)+\ve\leq\I[a_j,x_j]+\ve\leq \I[a,x].\] for every $j$. A
subsequence of $\{\mu_j\}$ which we denote as the same sequence
converges weak-$*$ to $\mu\in \J^1_a$. By Corollary \ref{Cor:Jcreg}
$\J_a=\J^1_a$ and hence $\mu$ belongs to $\J_a$. By Lemma
\ref{Lem:ConvergenceCrossProductDual} $\mu_j(x_j)$ converge to
$\mu(x)$ and hence
\[\re\mu(x)+\ve\leq\I[a,x]\leq \re\mu (x),\]
a contradiction. This finishes the proof of part ii. and the proof
of the proposition.
\end{proof}


\section{Variational unconstrained $K$-minimizers and gauge functions for convex
cones}\label{Sec:KMinimizers}

\par Let $X$ be a normed linear space and $K$ a convex closed pointed
cone in $X$. Then $K$ induces a partial order $\preceq$ on $X$ which
is defined by the relation
\[x,\,y\in X,\, x\preceq y\,\,\,\,\textnormal{if and only if} \,\,\,\,y-x\in K.\]
We will also consider the following order $\prec$ induced by the
interior $\ic K$ of $K$ in $X$:
\[x,\,y\in X,\, x\prec y\,\,\,\,\textnormal{if and only if} \,\,\,\,y-x\in \ic K.\]

\par Let $E$ be a normed space and $\Om$ be a subset of $E$. Often one
is interested in minimizing in the sense of this order a function
$F:\Om\to X$, that is, find a point $\alpha\in\Om$ such that there
exists no other $\beta\in\Om$ with $F(\beta)\preceq F(\alpha)$ and
$F(\beta)\not =F(\alpha)$. This is the problem of finding an
\emph{unconstrained} $K$-\emph{minimizer} of $F$ on $\Om$. Although
in the original definitions $E$ is considered to be a finite
dimensional space, our discussions in this section is a
straightforward extension to infinite dimensional setting.

\par We define the \emph{positive polar cone} $K^+$ of $K$ as the set
\[K^+=\{x^*\in X^*:x^*(x)\geq 0 \,\,\,\,\textnormal{for every}\,\,\,\,x\in K\}.\]
Let $C\sbs K^+$ be a weak-$*$ compact set which generates $K^+$ in
the following sense:
\[K^+=\ovr{\co } \cup_{t\geq 0}tC.\]
A \emph{gauge function} for $K$ is then defined as the function
$G:X\to \mathbb{R}$ by
\[G(x)=\sup_{x^*\in C}x^*(x).\]
It is clear that $G$ is a continuous sublinear functional. Gauge
function is essential for defining the $K$-\emph{steepest descent
direction} when the interior of $K$ is nonempty and one considers
the problem of finding a $K$-\emph{minimizer} of a continuously
differentiable function $F$ (see \cite{DrummondGranaIusem04},
\cite{FliegeGranaDrummond09}, \cite{FliegeSvaiter00},
\cite{DrummondSvaiter05}). We follow in this section the exposition
in \cite{DrummondSvaiter05}) where the case $E=\Om =\R^n$ was
considered.

\par In classical optimization (single-objective) $E=\R^n$, $X=\mathbb{R}$, $K=\mathbb{R}_+$,
the set of nonnegative real numbers and one can take $C=\{1\}$. For
the multi-objective optimization $E=\R^n$, $X=\mathbb{R}^m$, $m\geq
2$, $K$ and $K^+$ are the positive orthant of $\mathbb{R}^m$ and we
may take $C$ as the canonical basis of $\mathbb{R}^m$. For an
arbitrary closed pointed convex cone $K$ in $X$, the weak-$*$
closure in $X^*$ of the set $C=\{x^*\in K^+:\|x^*\|=1\}$ can be
used.

\par Given a point $\al\in \Om$ we define $f_{\al}:E\to\R$ as
\[f_{\al}(\nu):=G(DF(\al)\nu)=\sup\{x^*(DF(\al)\nu):x^*\in C\}\]
for any $\nu\in E$, where $DF(\al):E\to X$ is the differential of
$F$ at the point $\al$. Following \cite{DrummondSvaiter05} we say
that a vector $\nu\in E$ is a \emph{$K$-descent direction} at a
point $\al\in\Om$ if $f_{\al}(\nu)<0$. It is a well-known fact (see
\cite{Luc89}) that if $\nu\in E$ is a descent direction at a point
$\al\in\Om$, then there exists a number $t_0>0$ so that
\[F(\al+t\nu)\prec F(\al)\,\,\,\,\,\,\textnormal{for all}\,\,\,\,t\in (0,t_0).\]
We say that a point $\al\in\Om$ is \emph{$K$-critical} if there is
no $K$-descent direction at $\al$. That is to say, $\al$ is
$K$-critical if $f_{\al}(\nu)\geq 0$ for every $\nu\in E$. Note that
$\al$ is $K$-critical if and only if
\[-\ic K\cap \textnormal{Image}(DF(\al))=\emptyset.\]
The \emph{$K$-steepest descent direction} $\nu[\al]$ for $F$ at
$\al\in\Om$ is the solution of
\[\min \,\, f_{\al}(\nu)+(1/2)\|\nu\|^2,\,\,\,\,\,\,\nu\in E.\] The
optimal value of this problem will be denoted by $m[\al]$. Note that
the function $\nu\mapsto f_{\al}(\nu)$ is real-valued closed convex,
therefore, $\nu[\al]$ and $m[\al]$ are uniquely determined.
Moreover, the maps
\[(\al,\nu)\mapsto f_{\al}(\nu),\,\,\,\,\,\,\al\mapsto \nu[\al],\,\,\,\,\,\,\textnormal{and}\,\,\al\mapsto m[\al]\]
are continuous (see \cite[Lemma 3.3]{DrummondSvaiter05}).

\par We will now consider a variational problem of unconstrained
minimizers related to convex closed cones. Let $A$ be a metric
space. For every $a\in A$ let $K_a$ be a convex closed pointed cone
in $X$. Let $\J_a\sbs K^+_a$ be a set which generates $K^+_a$. Now
we consider the function $G:A\times X\to \mathbb{R}$ defined by
\[G(a,x)=\sup\{x^*(x):x^*\in\J_a\}.\]
Clearly the function $G(a,x)$ is continuous in the variable $x$ when
the first variable $a\in A$ is fixed. We are interested in
determining exact conditions which guarantee the continuity of the
variational Gauge function $G(a,x)$. With the notation of section
\ref{Sec:LowerEnvelopes} we have the relation
\[\I[-x,\J](a)=-G(a,x)\] for every $x\in X$ and $a\in A$.

\par Assuming certain very reasonable properties $\J 0$-$\J 3$ on the
sets $\J _a$ we get necessary and sufficient conditions in terms of
$\J_a$ for the function $G(a,x)$ to be upper or lower semicontinuous
or just to be continuous using Corollary \ref{Cor:Jcreg} and
Proposition \ref{Prop:JointContinuity}. These properties are
satisfied for example when the set $\cup_{a\in A}\J_a$ is bounded in
$X^*$. As a consequence we obtain the following result.

\bT\label{Th:GaugeFunctionContinuous} Let $A$ be a metric space, $X$
be a normed linear space, $K_a$ be a convex closed pointed cone in
$X$ and let $\J_a\sbs K^+_a$ be a set which generates $K^+_a$ for
every $a\in A$. Suppose that the properties $\J 0$-$\J 3$ are
satisfied. We have the following statements:

a. $G(a,x)$ is lower semicontinuous on $A\times X$ if and only if
$\J^{2cs}_a=\J_a$ for every $a\in A$;

b. $G(a,x)$ is upper semicontinuous on $A\times X$ if and only if
$\J_a=\J ^1_a$ for every $a\in A$;

c. $G(a,x)$ is continuous on $A\times X$ if and only if $A$ is
$\J$-$c$-regular. \eT


\par Let us go one step further. Let $F_s:E\to X$, $s\in S$, be a family of
continuously differentiable functions indexed by a topological space
$S$. When can one find a continuous selection of steepest
$K_a$-descent directions? We would like to establish some conditions
in terms of the generating sets $\J_a$ which guarantee the
continuity of the functions
\[(a,s,\al)\mapsto \nu [a,s,\al],\,\,\,\,\,\,\textnormal{and}\,\,(a,s,\al)\mapsto m [a,s,\al].\]
Here we denote the steepest $K_a$-descent direction for $F_s$ by
$\nu [a,s,\al]$ and the corresponding optimal value by $m
[a,s,\al]$. If we want to be more precise and want to emphasize the
involvement of the functions $F_s$ in these notations we will write
$\nu [a,s,\al;F_s]$ or $m [a,s,\al;F_s]$ respectively. The following
result which follows from Theorem \ref{Th:GaugeFunctionContinuous}
answers the question.

\bT Let $A$, $X$, $K_a$,  and $\J_a$ be as in Theorem
\ref{Th:GaugeFunctionContinuous}. Let $F_s:\R^n\to X$, $s\in S$, be
a family of continuously differentiable functions so that the
mapping
\[s\mapsto DF_s(\al)\,\,\,\,\,\,\textnormal{from}\,\,\,\,\,\,S\to
L(\R^n,X)\] is continuous for every $\al\in \R^n$. If $A$ is
$\J$-$c$-regular, then the mappings
\[(a,s,\al)\mapsto \nu [a,s,\al],\,\,\,\,\,\,\textnormal{and}\,\,(a,s,\al)\mapsto m
[a,s,\al]\] are continuous. \eT

\begin{proof} Note that by our assumption the map
\[(s,\al,\nu)\mapsto DF_s(\al)\nu\,\,\,\,\,\,\textnormal{from}\,\,\,\,\,\,S\times
\R^n\times \R^n\to \R^n\] is continuous. Since $A$ is
$\J$-$c$-regular, the gauge function $G(a,x)$ is continuous on
$A\times X$ by Theorem \ref{Th:GaugeFunctionContinuous}. Hence the
map
\[\kappa [a,s,\al ,\nu]:= G[a,DF_s(\al)\nu]+(1/2)\|\nu\|^2\,\,\,\,\,\,\textnormal{from}\,\,\,\,\,\,A\times S\times \R^n\times
\R^n\to\mathbb{R}\] is continuous. Let $(a_0,s_0,\al_0)$ be a point
in $A\times S\times \R^n$, $\nu_0=\nu [a_0,s_0,\al_0]$, and let
$m_0=m [a_0,s_0,\al_0]$. Our proof relies on the following observations:\\

\noindent \textit{Claim: Given  $\ve
>0$, there is an open
neighborhood $U$ of $(a_0,s_0,\al_0)$ in $A\times S\times \R^n$ so
that
\[G[a,DF_s(\al)\nu]+(1/2)\|\nu\|^2>G[a,DF_{s}(\al)\nu_0]+(1/2)\|\nu_0\|^2\]
for every $\nu\in \R^n$ with $\|\nu-\nu_0\|= \ve$ and for every
$(a,s,\al)\in U$}.
\smallskip

\noindent \textit{Proof of Claim:  } Let us assume the contrary. So
there exist $\ve>0$, vectors $\nu_k\in \R^n$ with
$\|\nu_k-\nu_0\|=\ve$ and points $(a_k,s_k,\al_k)$ which converge to
$(a_0,s_0,\al_0)$ so that
\[G[a_k,DF_{s_k}(\al_k)\nu_k]+(1/2)\|\nu_k\|^2\leq G[a_k,DF_{s_k}(\al_k)\nu_0]+(1/2)\|\nu_0\|^2\]
for every $k$. Since the set $\{\nu\in\R^n:\|\nu-\nu_0\|=\ve\}$ is
compact we may assume without loss of generality (refining
$\{\nu_k\}$ if necessary) that the vectors $\nu_k$ converge to a
vector $\nu '\in\R^n$. From the continuity of $G$ we have
\[\lim _kG[a_k,DF_{s_k}(\al_k)\nu_k]=G[a_0,DF_{s_0}(\al_0)\nu '].\]
Hence
\[G[a_0,DF_{s_0}(\al_0)\nu ']+(1/2)\|\nu '\|^2\leq m_0.\] Since
$m_0$ is the minimum value of the objective function $\kappa
[a_0,s_0,\al_0,\nu]$ and $\nu_0$  is the unique vector in $\R^n$
which minimizes this objective function, $\nu '=\nu_0$, which is a
contradiction since $\|\nu '-\nu_0\|=\ve$. Thus we have proved the
claim.

\par To finish the proof of the theorem let $\ve >0$ be given, and
let $U$ be the open set found above in the claim. Take any point
$(a,s,\al)\in U$. Let $\nu '=\nu [a,s,\al]$ and $k(\nu):=\kappa
[a,s,\al,\nu]$ for any vector $\nu\in\R^n$. We will show that
$\|\nu-\nu '\|<\ve $. Suppose to argue by contradiction that
$\|\nu-\nu '\|\geq\ve $. We can find a vector $\eta\in\R^n$ and a
number $0\leq t<1$ so that
\[\|\eta\|=\ve,\,\,\,\,\,\,  \textnormal{and} \,\,\,\,\,\,
\nu_0 +\eta =t\nu _0+(1-t)\nu '.\] Using the inequality proved in
the claim we have
\[k(\nu _0)<k(\nu_0+\eta)\leq tk(\nu _0)+(1-t)k(\nu ').\]
Thus we obtain $k(\nu _0)<k(\nu ')$ which is clearly a contradiction
to the fact that $\nu '$ is the minimizing vector of the function
$k(\nu)$ in $\R^n$. Therefore $\|\nu-\nu '\|<\ve $. The proof is
finished.
\end{proof}




\end{document}